\documentclass[12pt]{amsart}

\usepackage{amssymb,mathrsfs,amsmath,amsthm,color,bm,mathtools,bbm,wasysym}

\usepackage{enumerate}
\usepackage[centering]{geometry}
\allowdisplaybreaks

\theoremstyle{plain}
\newtheorem{thm}{Theorem}[section]

\newtheorem{lem}[thm]{Lemma}

\theoremstyle{remark}
\newtheorem{rem}{Remark}
\numberwithin{equation}{section}

\DeclareMathOperator{\hdim}{\dim_H}

\newcommand{\dif}{ \, \mathrm d}
\newcommand{\rc}{\mathcal R}

\newcommand{\N}{\mathbb N}
\newcommand{\R}{\mathbb R}

\newcommand{\lm}{\mathcal L}

\newcommand{\br}{\mathbf{r}}

\newcommand{\bx}{\mathbf{x}}
\newcommand{\by}{\mathbf{y}}

\newcommand{\qaq}{\mathrm{\quad and\quad}}

\begin{document}
	\title[Quantitative recurrence properties and strong Borel-Cantelli lemma]{Quantitative recurrence properties and strong dynamical Borel-Cantelli lemma for dynamical systems with exponential decay of correlations}
	\author{Yubin He}
\thanks{This work was supported by NSFC (Nos. 1240010704).}
	\address{Department of Mathematics, Shantou University, Shantou, Guangdong, 515063, China}

	\email{ybhe@stu.edu.cn}

%
%

	\subjclass[2020]{37A05, 37B20}

	\keywords{recurrence, strong dynamical Borel-Cantelli lemma, exponential decay of correlations.}
	\begin{abstract}
		Let $ ([0,1]^d,T,\mu) $ be a measure-preserving dynamical system so that the correlations decay exponentially for H\"older continuous functions. Suppose that $ \mu $ is absolutely continuous with a density function $ h\in L^q(\lm^d) $ for some $ q>1 $, where $ \lm^d $ is the $ d $-dimensional Lebesgue measure. Under mild conditions on the underlying dynamical system, we obtain a strong dynamical Borel-Cantelli lemma for recurrence: For any sequence $ \{R_n\} $ of hyperrectangles with sides parallel to the axes and centered at the origin,
		\[\sum_{n=1}^{\infty}\lm^d(R_n)=\infty\quad\Longrightarrow\quad\lim_{n\to\infty}\frac{\sum_{k=1}^{n}\mathbbm{1}_{R_k+\bx}(T^k\bx)}{\sum_{k=1}^{n}\lm^d(R_k)}=h(\bx)\quad\text{for $ \mu $-a.e.\,$ \bx $},\]
		where $ \bx\in[0,1]^d $ and $ R_k+\bx $ is the translation of $ R_k $. The result applies to Gauss map, $\beta$-transformation and expanding toral endomorphisms.
	\end{abstract}
	\maketitle

\section{Introduction}
Let $ (X,d,T,\mu) $ be a probability measure-preserving system endowed with a compatible metric $ d $ so that $ (X,d) $ is complete and separable. A corollary of the well-known  Poincar\'e's recurrence theorem (see e.g. \cite[Theorem 3.3]{Fu81}) asserts that $ \mu $-almost all points $ x\in X $ are recurrent, i.e.
\[\liminf_{n\to\infty} d(T^nx,x)=0.\]
This is a qualitative statement in nature. The first quantitative result was given by Boshernitzan \cite{Boshernitzan93}.
\begin{thm}[{\cite[Theorem 1.2]{Boshernitzan93}}]
	Let $ (X,d,T,\mu) $ be a probability measure-preserving system endowed with a metric $ d $. Assume that for some $ \alpha>0 $ the $ \alpha $-Hausdorff measure $ \mathcal H^\alpha $ is $ \sigma $-finite on $ (X,d) $. Then for $ \mu $-almost every $ x\in X $, we have
	\[\liminf_{n\to\infty} n^{1/\alpha}d(T^nx,x)<\infty.\]
	Moreover, if $ \mathcal H^\alpha(X)=0 $, then for $ \mu $-almost every $ x\in X $, we have
	\[\liminf_{n\to\infty} n^{1/\alpha}d(T^nx,x)=0.\]
\end{thm}
Further improvements were established by Barreira and Saussol \cite{BarreiraSaussol01} who related the recurrence rate to the local pointwise dimension of $\mu$.

Boshernitzan's result can be reformulated as: For $ \mu $-almost every $ x $, there is a constant $ c(x)>0 $ such that
\begin{equation}\label{eq:ref1}
	d(T^nx,x)<c(x)n^{-1/\alpha}\quad\text{ for infinitely many }n\in\N,
\end{equation}
or equivalently
\begin{equation}\label{eq:ref2}
	\sum_{n= 1}^\infty\mathbbm{1}_{B(x,c(x)n^{-1/\alpha})}(T^nx)=\infty,
\end{equation}
where $ \mathbbm{1}_E $ denotes the indicator function of the set $ E $ and $B(x,r)$ denotes the open ball with center $x\in X$ and radius $r>0$.
The first reformulation \eqref{eq:ref1} leads to understanding the $\mu$-measure of the recurrence set
\begin{equation}\label{eq:recurrence set ball}
	\rc(\{r_n\}):=\{x\in X: d(T^nx,x)<r_n\text{ for infinitely many } n\in\N\},
\end{equation}
where $ \{r_n\} $ is a sequence of non-negative real numbers. For some rapidly mixing dynamical systems (see \cite{BakerFarmer21,BK24,CWW19,HeLi24,HLSW22,KKP21,KleinbockZheng22}),  there have been verified that $\mu(\rc(\{r_n\})) $ obeys a zero-one law according as the convergence or divergence of the series $ \sum_{n=1}^{\infty}r_n^{\hdim X} $, where $ \hdim $ stands for the Hausdorff dimension. Such a zero-one law is also referred to as {\em dynamical Borel-Cantelli lemma}.

The second reformulation \eqref{eq:ref2} leads to further quantify the divergence rate of the sequence with respect to $ n $
\[\sum_{k=1 }^n\mathbbm{1}_{B(x,r_k)}(T^kx).\]
It is reasonable to expect that the stronger statement would hold: If the system mixes sufficiently fast, then
\begin{equation}\label{eq:expsbc}
	\sum_{k=1 }^n\mathbbm{1}_{B(x,r_k)}(T^kx)\approx c(x)\sum_{k=1}^nr_k^{\hdim X}\quad\text{for $ \mu $-almost every $ x $},
\end{equation}
where $ c(x) $ is a constant dependent of $ x $. Such results are called {\em strong dynamical Borel-Cantelli lemma}. Apparently, if the sum $ \sum_{n=1}^\infty r_n^{\hdim X} $ diverges and \eqref{eq:expsbc} holds, then one can deduce that $ \mu(\rc(\{r_n\}))=1 $. In \cite{Per23}, under some mild conditions, Persson proved a strong dynamical Borel-Cantelli lemma similar to \eqref{eq:expsbc} for a class
of dynamical systems with exponential decay of correlations on the unit interval. This result was subsequently generalized by Sponheimer \cite{Spo23} to more general setting including Axiom A diffeomorphism. Notably, both results require the sequence $ \{r_n\} $ to satisfy
\begin{equation}\label{eq:radcon}
	\lim_{a\to 1^+}\limsup_{n\to\infty}\frac{r_n}{r_{an}}=1\quad\text{and}\quad r_{n}\ge \frac{(\log n)^{4+\varepsilon}}{n}\text{ for some $ \varepsilon>0 $}.
\end{equation}
Naturally, one would like to prove \eqref{eq:expsbc} without assuming \eqref{eq:radcon}, as \eqref{eq:expsbc} with $ x $ replaced by a fixed center $ y\in X $ has been verified in various rapidly mixing systems, see \cite{ChKl01,Kim07,LLVZ23,Phi67}. To this end, we need to impose stronger  conditions on the underlying dynamical system than those in \cite{Per23,Spo23}.

Let us begin with some notation. Let $ X $ be the unit cube $ [0,1]^d $ endowed with the maximum norm $ |\cdot| $. The closure, boundary, $\varepsilon$-neighbourhood and cardinal number of $ A $ will be denoted by $ \overline A $, $ \partial A $, $ A_\varepsilon $ and $  \#  A $, respectively.  Let $ C^{\theta}([0,1]^d) $ be the space of $ \theta $-H\"older continuous functions endowed with the $ \theta $-H\"older norm
\[\|f\|_{\theta}:=\max_{\bx\in[0,1]^d}|f(\bx)|+\sup\bigg\{\frac{|f(\bx)-f(\by)|}{|\bx-\by|^{\theta}}:\bx,\by\in[0,1]^d\bigg\}.\]

Let $ \{U_i\}_{i\in \mathcal I} $ be a countable family of pairwise disjoint open subsets in $ [0,1]^d $ with $ \bigcup_{i\in\mathcal I}\overline{U_i}=[0,1]^d $. The following conditions on the measure-preserving dynamical system $ ([0,1]^d,T,\mu) $ will play a central role in our work.


\noindent \textit{Condition I}. (Exponential decay of correlations) Given $ 0<\theta\le 1 $. There exist constants $ \tau=\tau(\theta)>0 $ and $ C=C(\theta)>0 $ such that for all $ f\in C^\theta([0,1]^d) $, $ g\in L^1(\mu) $ and all $ n\ge 0 $
\begin{equation}\label{eq:exp dec}
	\bigg|\int f\cdot g\circ T^n\dif \mu-\int f\dif \mu\int g \dif \mu\bigg|\le Ce^{-\tau n}\|f\|_\theta|g|_1,
\end{equation}
where $|g|_1=\int|g|\dif\mu$.

\noindent \textit{Condition II}. The  $ T $-invariant measure $ \mu $ is absolute continuous with respect to the $d$-dimensional Lebesgue measure $ \lm^d $ for which the density function $ h $ belongs to $ L^q(\lm^d) $ for some $ q>1 $.

\noindent \textit{Condition III}. There exists $ L>0 $ such that for any $ i\ge 1 $ and any $ \bx,\by\in \overline{U_i} $,
\begin{equation}\label{eq:Lipconst}
	|T\bx-T\by|\ge L|\bx-\by|.
\end{equation}

\noindent \textit{Condition IV}. There exist constants $ 0<\alpha\le d $, $\varepsilon_0>0$ and $ K_1>0 $ such that
\[\sup_{i\in\mathcal I}\sup_{0<\varepsilon<\varepsilon_0}\frac{\lm^d\big((\partial U_i)_\varepsilon\big)}{\varepsilon^\alpha}<K_1.\]

\noindent \textit{Condition V}. There exist a set $ \mathcal A\subset [0,1]^d $, and constants $ \beta_1,\beta_2>0 $ and $ K_2>0 $ such that for any $ r<1 $,
\[\#\bigg\{i\in \mathcal I:U_i\nsubseteq \bigcup_{x\in \mathcal A} B(x,r)\bigg\}\le K_2r^{-\beta_1}\]
and
\[\lm^d\bigg(\bigcup_{x\in \mathcal A} B(x,r)\bigg)\le K_2r^{\beta_2}.\]

\begin{rem}
	\
	(1) Condition I is satisfied by many dynamical systmes, such as Gauss map \cite{Phi67}, $\beta$-transformation \cite{Phi67} and expanding toral endomorphisms \cite[Theorem 2.3]{Baladi00}.

	(2) In \cite[Theorem 2.6]{ABB22}, Allen, Baker and B\'ar\'any showed that for non-uniform Bernoulli measure on the shift space, the convergence/divergence of the natural volume sum does
	not always determine the measure of a recurrence set. Consequently, to establish the (strong) dynamical Borel-Cantelli lemma, it is essential that $\mu$ be an absolutely continuous measure with respect to $\lm^d$. Additionally, the technical conditions $h\in L^q(\lm^d)$ with $q>1$ is needed to effectively estimate the measure of certain sets.

	(3) Although Condition III does not assume that $T$ is expanding, i.e. $L>1$, to the author's best knowledge, systems satisfying Condition I are usually expanding.

	(4) Condition IV says that the boundaries of elements in the partition $ \{U_i\}_{i\in\mathcal I} $ are sufficiently regular. If $ \mathcal I $ is finite and the boundary of each $ U_i $ is included in a $ C^1 $ piecewise embedded compact submanifold of codimension one, then Condition IV holds with $ \alpha=1 $. This condition arises in a natural reason. What needs to be dealt with is the correlations between sets, while Condition I only gives estimate of the correlations between functions. Therefore, we need to approximate these sets with H\"older continuous functions, which leads to condition IV.

	(5) Condition V says that most elements in $ \{U_i\}_{i\in\mathcal I} $ will concentrate on a set with small Lebesgue measure. If $ \mathcal I $ is finite, then we can take $ \mathcal A=\emptyset $ and Condition V holds with arbitrary positive $ \beta_1 $ and $ \beta_2 $. For the Gauss map, the associated partition $ \{U_i\}_{i\in\mathcal I} $ is infinite. We can take $ \mathcal A=\{0\} $, and Condition V holds with $ \beta_1=\beta_2=1 $.

	(6) It is worth pointing out that although Conditions I--V are seemingly restrictive, there are still many dynamical systems satisfy these conditions, such as Gauss map, $\beta$-transformation and expanding toral endomorphisms.
\end{rem}

For any $ \bx=(x_1,\dots,x_d)\in [0,1]^d $ and $ \br=(r_1,\dots,r_d)\in(\R_{\ge 0})^d $, let
\[R(\bx,\br):=\prod_{i=1}^{d}[x_i-r_i,x_i+r_i]\]
be a hyperrectangle. Motivated by the weighted theory of Diophantine approximation, for a sequence $ \{\br_n\} $ of vectors with $ \br_n=(r_{n,1},\dots,r_{n,d})\in(\R_{\ge 0})^d $, we define
\begin{equation}\label{eq:rectangle}
	\rc(\{\br_n\})=\{\bx\in [0,1]^d:T^n\bx\in R(\bx,\br_n)\text{ for infinitely many $ n\in\N $}\}.
\end{equation}
If all the entries of $ \br_n $ coincide, that is $ r_{n,1}=\cdots=r_{n,d}=r_n $, then $ \rc(\{\br_n\}) $ is nothing but $\rc(\{r_n\}) $ defined in \eqref{eq:recurrence set ball}.
\begin{thm}\label{t:main}
	Suppose that the measure preserving dynamical system $ ([0,1]^d,T,\mu) $ satisfies Conditions I--V. Then,
	\[\sum_{n=1}^{\infty} r_{n,1}\cdots r_{n,d}<\infty\quad\Longrightarrow\quad\mu(\rc(\{\br_n\}))=0.\]
	Moreover, if $\lim_{n\to\infty}|\br_n|=0$ and $ \sum_{n=1}^{\infty} r_{n,1}\cdots r_{n,d}=\infty $, then for $ \mu $-almost every $ \bx $,
	\begin{equation}\label{eq:sbc}
		\lim_{n\to\infty}\frac{\sum_{k=1}^n \mathbbm 1_{R(\bx,\br_k)}(T^k\bx)}{\sum_{k=1}^n2^dr_{k,1}\cdots r_{k,d}}=h(\bx),
	\end{equation}
	where $h$ is the density function given in Condition II.
\end{thm}

\begin{rem}
	The restriction of $X$ to $[0,1]^d$ is not necessary. In fact, if $X$ is a compact subset of $\R^d$ with smooth boundaries, and $(X,T,\mu)$ satisfies Conditions I--V, then we can define a new map $\tilde T:[0,1]^d\to[0,1]^d$ as follows: For any $\bx\in[0,1]^d$,
	\[\tilde T(\bx)=\begin{cases}
		\bx&\text{if $D\bx\notin X$},\\
		T(D\bx)/D&\text{if $D\bx\in X$},
	\end{cases}\]
	where $D=\max\{|\by|:\by\in X\}$. Theorem \ref{t:main} then applies to $([0,1]^d,\tilde T,\tilde T_*\mu)$ and hence to $(X,T,\mu)$, where $\tilde T_*\mu$ is the pushforward of $\mu$.
\end{rem}
\begin{rem}
	In \cite{HeLi24}, the author and Liao proved that under conditions similar to Conditions I--V, the $ \mu $-measure of $ \rc(\{\br_n\}) $ is either zero or one according as the series $ \sum_{n=1}^{\infty}r_{n,1}\cdots r_{n,d} $ converges or not. However, the results in \cite{HeLi24} are not applicable to the Gauss map, as the authors required the partition $ \{U_i\}_{i\in\mathcal I} $ to be finite. Such assumption is not assumed in the present work, and the conclusion is stronger than theirs. Our methods also work under the setting of  \cite{HeLi24}, whence the strong dynamical Borel-Cantelli lemma applies to a broad class of expanding matrix transformations, not limited to expanding toral endomorphisms.
\end{rem}
\begin{rem}
	The proof of Theorem \ref{t:main} partially make use of the geometric properties of Euclidean space that do not necessarily hold for other metric spaces. For example, any hypercube can be partitioned into smaller hypercubes of equal length, and the measure of arbitrary annuli $B(\bx,r+\varepsilon)\setminus B(\bx,r)$ is bounded by $\varepsilon$ up to a multiplicative constant, and so on. The same properties are also implicitly used in \cite{HeLi24}. Our methods are applicable when the ambient metric space meets similar geometric properties. However, the only non-trivial examples we have known are self-conformal sets on $\R$ that satisfy the open set condition. In these cases, the proof of the corresponding strong dynamical Borel-Cantelli lemma is considerably simpler and more straightforward. To maintain readability, we do not pursue such a generalization.
\end{rem}

\section{Correlations of sets with regular boundary}
In this section, we give correlation estimates for sets with regular boundaries. We begin with a simple estimation on measurable subsets which will be frequently referenced in the subsequent discussion.

Throughout, for positive-valued functions $f$ and $g$, we write $f\ll g$ if there exists a constant $c>0$ such that $f(x)\le cg(x)$ for all $x$. We write $f\asymp g$ if $f\ll g\ll f$.
\begin{lem}\label{l:hqs}
	Let $ s=1-1/q $. For any measurable set $ F $, we have
	\[\mu(F)\le |h|_q\lm^d(F)^s,\]
	where $ |h|_q=(\int |h|^q\dif\lm^d)^{1/q} $ is the $ L^q $-norm of $ h $.
\end{lem}
\begin{proof}
	By H\"older inequality,
	\[\mu(F)=\int \mathbbm{1}_F\cdot h\dif\lm^d\le|h|_q\lm^d(F)^s.\qedhere\]
\end{proof}
The following lemma enables us to estimate correlations for sets beyond H\"older continuous functions.
\begin{lem}\label{l:EcapF}
	Let $ E $ and $ F $ be two measurable sets. Then for $ 0<\varepsilon<1 $,
	\[\mu(E\cap T^{-n}F)\le \mu(F)\big(\mu(E)+|h|_q\lm\big((\partial E)_\varepsilon\big)^s+2Ce^{-\tau n}/\varepsilon\big)\]
	and
	\[\mu(E\cap T^{-n}F)\ge \mu(F)\big(\mu(E)-|h|_q\lm^d\big((\partial E)_\varepsilon\big)^s-Ce^{-\tau n}/\varepsilon\big).\]
\end{lem}
\begin{proof}
	Let $ 0<\varepsilon<1 $. Define
	\[f_\varepsilon(x)=\begin{dcases}
		1-\frac{d(x,E)}{\varepsilon}&\text{if $ d(x,E)<\varepsilon $},\\
		0&\text{if $ d(x,E)\ge \varepsilon $}.
	\end{dcases}\]
	It is not difficult to verify that $ f_\varepsilon $ is Lipschitz with Lipschitz constant $ 1/\varepsilon $. Hence, $ f_\varepsilon\in C^\theta([0,1]^d) $. Applying the exponential decay of correlations (see Condition I) to $ f=f_\varepsilon $ and $ g=\mathbbm{1}_F $, we have
	\[\begin{split}
		\int f_\varepsilon\cdot\mathbbm{1}_F\circ T^n\dif\mu&\le \int\mathbbm{1}_F\dif\mu\cdot \bigg(\int f_\varepsilon\dif\mu+Ce^{-\tau n}\|f_\varepsilon\|_\theta\bigg)\\
		&\le \int\mathbbm{1}_F\dif\mu \cdot\bigg(\int f_\varepsilon\dif\mu+Ce^{-\tau n}(1+1/\varepsilon)\bigg).
	\end{split}\]
	Note that $ f_\varepsilon(x)=1 $ for $ x\in E $ and the support of $ f_\varepsilon $ is contained in $ E_\varepsilon= E\cup (\partial E)_\varepsilon $. Thus, we have
	\[\int f_\varepsilon\dif\mu\le \mu(E_\varepsilon)\le \mu(E)+\mu\big((\partial E)_\varepsilon\big),\]
	and so
	\[\begin{split}
		\mu(E\cap T^{-n}F)&\le \int f_\varepsilon\cdot\mathbbm{1}_F\circ T^n\dif\mu\le \mu(F)\big(\mu(E)+\mu\big((\partial E)_\varepsilon\big)+Ce^{-\tau n}(1+1/\varepsilon)\big)\\
		&\le \mu(F)\big(\mu(E)+\mu\big((\partial E)_\varepsilon\big)+2Ce^{-\tau n}/\varepsilon\big).
	\end{split}\]
Apply Lemma \ref{l:hqs}, the first inequality follows.

To prove the other inequality, by letting $ E_{-\varepsilon}=E\setminus (\partial E)_\varepsilon $, we see that
\[x\in E_{-\varepsilon}\quad\Longrightarrow\quad B(x,\varepsilon)\subset E.\]
Define
\[g_\varepsilon(x)=\begin{dcases}
	1-\frac{d(x,E_{-\varepsilon})}{\varepsilon}&\text{if $ d(x,E_{-\varepsilon})<\varepsilon $},\\
	0&\text{if $ d(x,E_{-\varepsilon})\ge \varepsilon $}.
\end{dcases}\]
Clearly, by definition
\[\mathbbm{1}_{E_{-\varepsilon}}(x)\le g_\varepsilon(x)\le \mathbbm{1}_E(x).\]
Now employing the argument similar to above, we arrive at the second inequality of the lemma.
\end{proof}
As a consequence, we get the exponential decay of correlations for hyperrectangles.
\begin{lem}\label{l:RcapF}
	Let $ R $ be a hyperrectangle and let $ F $ be a measurable set. Then, we have
	\[|\mu(R\cap T^{-n}F)-\mu(R)\mu(F)|\le c_1\mu(F)e^{-\tau_1n},\]
	where $ c_1=4d|h|_q+2C $ and $ \tau_1=s\tau/(1+s) $.
\end{lem}
\begin{proof}
	Note that for any $ \varepsilon>0 $, $ (\partial R)_\varepsilon $ is contained in $ 2d $ hyperrectangles, each of which has volume less than $ 2\varepsilon $. Applying Lemma \ref{l:EcapF} with $ \varepsilon=e^{-\tau n/(1+s)} $, we have
	\[\mu(R\cap T^{-n}F)\le \mu(F)(\mu(R)+4d|h|_qe^{-s\tau n/(1+s)}+2Ce^{-s\tau n/(1+s)})\]
	and
	\[\mu(R\cap T^{-n}F)\ge \mu(F)(\mu(R)-4d|h|_qe^{-s\tau n/(1+s)}-Ce^{-s\tau n/(1+s)}),\]
	which finishes the proof.
\end{proof}
The proof of divergence part of the theorem crucially relies on the following two technical lemmas. The ideas originate from \cite[Lemmas 2.8 and 2.10]{HeLi24}, but the proofs are different because, in the current setting, $ T $ is not necessarily expanding and the partition $ \{U_i\}_{i\in \mathcal I} $ may be infinite. For any $n\ge 1$, define
\[	\mathcal F_n:=\{U_{i_0}\cap\cdots\cap T^{-(n-1)}U_{i_{n-1}}:i_0,\dots,i_{n-1}\in\mathcal I\}.\]
\begin{lem}\label{l:JcRcR}
	There exists a constant $ c_2 $ such that for any $0<\varepsilon<\varepsilon_0$ with $\varepsilon_0$ given in Condition IV and any hyperrectangles $ R_1,R_2\subset[0,1]^d $,
	\[\mu\big((\partial(J_n\cap R_1\cap T^{-n}R_2))_\varepsilon\big)\le c_2(\max(n, L^{-nd})\varepsilon^{\min(1,\alpha)})^{s}.\]
\end{lem}
\begin{proof}
	Let $0<\varepsilon<\varepsilon_0$. First, we claim that for any $ J_n\in\mathcal F_n $,
	\begin{equation}\label{eq:boundary of cylinder}
		\lm^d\big((\partial J_n)_\varepsilon\big)\ll \begin{cases}
			\dfrac{L^{-nd}-1}{L^{-d}-1}\varepsilon^\alpha&\text{if $L<1$},\\
			n\epsilon^\alpha&\text{if $L\ge 1$},
		\end{cases}
	\end{equation}
	where $0<\alpha\le d$ is given in Condition IV.
	We proceed by induction. For $ n=1 $, we have $\partial J_1=\partial U_i$ for some $i\in\mathcal I$, and so this is implied by Condition IV.

	Assume that \eqref{eq:boundary of cylinder} holds for some $ n\ge 1 $. We will prove that \eqref{eq:boundary of cylinder} holds for $ n+1 $. Write $ J_{n+1}=J_{n}\cap T^{-n}(U_{i_n}) $, where $ J_n\in\mathcal F_n $ and $ i_n\in\mathcal I $. By Condition III, the inverse of the map $ T^n|_{J_{n}}\colon J_n\to T^nJ_n $, denoted by $ f_n $, exists and satisfies
	\begin{equation}\label{eq:inverse bounded derivative}
		|f_n(\bx)-f_n(\by)|\le L^{-n}|\bx-\by| \quad\text{for any $ \bx,\by\in T^nJ_n $}.
	\end{equation}
	Since $ f_n $ is invertible and Lipschitz, we have
	\begin{align}
		\partial J_{n+1}&=\partial \big(J_{n}\cap T^{-n}(U_{i_n})\big)=\partial \big(f_n(T^nJ_{n}\cap U_{i_n})\big)\notag\\
		&=f_n\big(\partial(T^nJ_{n}\cap U_{i_n})\big).\label{eq:boundary inverse}
	\end{align}
	Clearly, $ \partial(T^nJ_{n}\cap U_{i_n})\subset \partial (T^nJ_n)\cup \partial U_{i_n} $. Thus, we get
	\[\begin{split}
		\partial J_{n+1}\subset f_n\big(\partial(T^nJ_n)\big)\cup f_n(\partial U_{i_n})=\partial J_n\cup f_n(\partial U_{i_n}).
	\end{split}\]
	Let $ \mathcal E $ be an $ \varepsilon $-net of $ \partial U_{i_n} $ with minimal cardinality. The minimal property implies that the balls with raduis $ \varepsilon/2 $ and center in $ \mathcal E $ are pairwise disjoint. Hence, by Condition IV, we have
	\[\#\mathcal E\cdot \varepsilon^d\asymp \lm^d\big((\partial U_{i_n})_{\varepsilon}\big)\ll \varepsilon^\alpha,\]
	which implies that
	\[\#\mathcal E\ll\varepsilon^{\alpha-d}.\]
	Given that $ f_n $ is Lipschitz with Lipschitz constant $ L^{-n} $, the image of $\mathcal E$ under $f_n$ is a $ L^{-n}\varepsilon $-net of $ f_n(\partial U_{i_n}) $.
	If $L<1$, then we have $L^{-n}>1$ and so $ (f_n(\partial U_{i_n}))_\varepsilon $ is contained in the union of balls with radius $ 2L^{-n}\varepsilon $ and centers in $f_n\mathcal E$. Thus,
	\begin{equation}\label{eq:upppar}
		\lm^d\big((f_n(\partial U_{i_n}))_\varepsilon\big)\ll \#f_n\mathcal E\cdot  (L^{-n}\varepsilon)^d\le \#\mathcal E\cdot  (L^{-n}\varepsilon)^d\ll L^{-nd}\varepsilon^\alpha.
	\end{equation}
	On the other hand, if $L\ge 1$, then $L^{-n}<1$ and so $f_n\mathcal E$ is also an $\varepsilon$-net of $f_n(\partial U_{i_n})$. Thus,
	\[\lm^d\big((f_n(\partial U_{i_n}))_\varepsilon\big)\ll \#f_n\mathcal E\cdot  \varepsilon^d\le \#\mathcal E\cdot  \varepsilon^d\ll \varepsilon^\alpha.\]

	In what follows, we may assume that $L<1$ as the argument for $L\ge 1$ is similar. Using the inductive hypothesis and \eqref{eq:upppar}, we  have
	\begin{align}
		\lm^d\big((\partial J_{n+1})_\varepsilon\big)
		&\le \lm^d\big((\partial J_n)_\varepsilon\big)+\lm^d\big((f_n(\partial U_{i_n}))_\varepsilon\big)\notag\\
		&\ll\frac{L^{-nd}-1}{L^{-d}-1}\varepsilon^\alpha+L^{-nd}\varepsilon^\alpha\label{eq:boundary upper bound}\\
		&=\frac{L^{-(n+1)d}-1}{L^{-d}-1}\varepsilon^\alpha\notag,
	\end{align}
	which proves the claim.

	Note that $\lm^d\big((\partial R)_\varepsilon\big)\ll \varepsilon$. By the same reason as \eqref{eq:upppar}, we may take $ \alpha=1 $ and deduce that
	\[\lm^d\big((f_{n}(\partial R_2))_\varepsilon\big)\ll L^{-nd}\varepsilon.\]
	Similarly, following the same lines as \eqref{eq:boundary upper bound} with $U_{i_n}$ replaced by $R_2$ and note that $ \alpha\ge 1 $ (see Condition IV), we have
	\[\begin{split}
		\lm^d\big((\partial (J_n\cap R_1\cap T^{-n}R_2))_\varepsilon\big)&\le \lm^d\big((\partial R_1)_\varepsilon\big)+\lm^d\big((\partial (J_n\cap T^{-n} R_2))_\varepsilon\big)\\
		&\le \lm^d((\partial R_1)_\varepsilon)+\lm^d\big((\partial J_n)_\varepsilon\big)+\lm^d\big((f_n(\partial R_2))_\varepsilon\big)\\
		&\ll\varepsilon+\frac{L^{-nd}-1}{L^{-d}-1}\varepsilon^\alpha+L^{-nd}\varepsilon\\
		&\ll L^{-nd}\varepsilon^{\min(1,\alpha)}.
	\end{split}\]
	Apply Lemma \ref{l:hqs}, and the proof is finished.
\end{proof}
We will use Condition IV and the last lemma to prove an estimation similar to \cite[Lemma 2.10]{HeLi24}. However, their methods cannot be directly applied to the current setting since the partition $ \{U_i\}_{i\in\mathcal I} $ may be infinite.  For $m<n$, let
\[D(m,n):=\bigcup_{i=0}^{m-1} \bigg(T^{-i}\bigcup_{x\in \mathcal A}B(x,n^{-7/(s\beta_2)})\bigg)\]
and
\[\begin{split}
	\mathcal F_m(n):=\bigg\{U_{i_0}\cap\cdots\cap T^{-(m-1)}U_{i_{m-1}}:U_{i_j}\nsubseteq \bigcup_{x\in \mathcal A}B(x,n^{-7/(s\beta_2)})\text{, $ 0\le j\le m-1$}\bigg\}.
\end{split}\]
 In other words, $D(m,n)$ denotes the set of points whose forward orbits, under the iteration of $T$ up to $m-1$, hit the set $\bigcup_{x\in \mathcal A}B(x,n^{-7/(s\beta_2)})$ at least once, while $\mathcal F_m(n)$ is slightly larger than the complement of $D(m,n)$ for technical reason. It can be checked that
 \[[0,1]^d\setminus D(m,n)\subset \bigcup_{J_m\in\mathcal F_m(n)}J_m.\]
Moreover, by Condition V, we have
\[\#\mathcal F_m(n)\le (K_2 n^{7\beta_1/(s\beta_2)})^m.\]
\begin{lem}\label{l:measure of intersection}
	There exists constants $ c_3 $ and $ 0<\eta<\tau_1 $ such that

	(1) for any $m<n$, we have $ \mu\big(D(m,n)\big)\le c_3n^{-6} $.

	(2) for any $ m<\sqrt n $, and any hyperrectangles $ R_1,R_2, R_3\subset [0,1]^d $, we have
	\[\sum_{J_m\in\mathcal F_m(n)}\mu(J_m\cap R_1\cap T^{-m}R_2\cap T^{-n}R_3)\le \big(\mu(R_1\cap T^{-m}R_2)+ c_3e^{-\eta n}\big)\mu(R_3),\]
	for some $0<\eta<\tau_1$, where $\tau_1$ is given in Lemma \ref{l:RcapF}.
\end{lem}
\begin{proof}
	(1) By Lemma \ref{l:hqs} and Condition V,
	\begin{align*}
		\mu(D(m,n))&\le m\mu\bigg(\bigcup_{x\in \mathcal A}B(x,n^{-7/(s\beta_2)})\bigg)\ll m\lm^d\bigg(\bigcup_{x\in \mathcal A}B(x,n^{-7/(s\beta_2)})\bigg)^s\\
		&\ll m\cdot n^{-7}\le n^{-6}.
	\end{align*}
	(2) Write $L_m=\max(m,L^{-dm})$ and $\kappa=\min(1,\alpha)$. Applying Lemmas \ref{l:EcapF} and \ref{l:JcRcR} with $ \varepsilon=e^{-\tau n/2} $, we have
	\begin{align}
		&\sum_{J_m\in\mathcal F_m(n)}\mu(J_m\cap R_1\cap T^{-m}R_2\cap T^{-n}R_3)\notag\\
		\le &\sum_{J_m\in\mathcal F_m(n)}(\mu(J_m\cap R_1\cap T^{-m}R_2)+c_2(L_me^{-\kappa\tau n/2})^s+c_1e^{-\tau n/2})\mu(R_3)\notag\\
		\le &\Big(\mu(R_1\cap T^{-m}R_2)+ \# \mathcal F_m(n)\big(c_2(L_me^{-\kappa\tau n/2})^s+c_1e^{-\tau n/2}\big)\Big)\mu(R_3).\label{eq:exponentiall decay intersection of three balls}
	\end{align}
	Since $ m\le \sqrt n $, both terms
	\[ \# \mathcal F_m(n)\le  (K_2 n^{7\beta_1/(s\beta_2)})^m\le e^{\sqrt n\log K_2+7\beta_1\sqrt n\log n/(s\beta_2)}\]
	and
	\[L_m=\max(m,L^{-dm})\le\max(\sqrt n, e^{-d\sqrt{n}\log L})\]
	grow at most subexponentially. Thus, there exist $ c_3>0 $ and $ 0<\eta<\tau_1 $ such that \eqref{eq:exponentiall decay intersection of three balls} is bounded by
	\[\big(\mu(R_1\cap T^{-m}R_2)+ c_3e^{-\eta n}\big)\mu(R_3).\qedhere\]
\end{proof}

\section{Proof of Theorem \ref{t:main}}
This section is dedicated to the proof of Theorem \ref{t:main}. To begin, assume that $ \sum_{n=1}^{\infty}r_{n,1}\cdots r_{n,d}<\infty $. In \cite[Proposition 2.1]{HeLi24}, the author and Liao established that $ \mu(\rc(\{\br_n\}))=0 $ under Condition II and the following exponential decay of correlations for hyperrectangles $ R_1 $ and $ R_2 $,
\[|\mu(R_1\cap T^{-n}R_2)-\mu(R_1)\mu(R_2)|\le \mu(R_2)c_1e^{-\tau_1n}.\]
Such property is also assumed in the present setting (see Condition I and Lemma \ref{l:RcapF}), so we can employ the same argument as in \cite[Proposition 2.1]{HeLi24} to conclude the convergence part of the theorem.

Now, suppose that
\begin{equation}\label{eq:div}
	\sum_{n=1}^{\infty}r_{n,1}\cdots r_{n,d}=\infty.
\end{equation}
In what follows, to simplify the notation, write
\begin{equation}\label{eq:gamman}
	\gamma_n=r_{n,1}\cdots r_{n,d}.
\end{equation}
Without loss of generality, we further assume that for all $ n\in\N $,
\begin{equation}\label{eq:assr=0}
	\text{either\quad $ \gamma_n=0 $\quad or\quad $ \gamma_n\ge n^{-2} $.}
\end{equation}
We will prove the divergence part of the theorem under these assumptions. Indeed, let
\[D=\{n\in\N: \gamma_n> n^{-2}\} \quad\text{and}\quad D_n=\{1,\dots, n\}\cap D. \]
Then, by \eqref{eq:div} and the fact that $ \sum_{n=1}^\infty n^{-2}<\infty $, we have
\[\lim_{n\to\infty}\dfrac{\sum_{k\in D_n}\gamma_k}{\sum_{k=1}^n\gamma_k}=1.\]
On the other hand, since
\[\sum_{k\in\N\setminus D}\gamma_k\le \sum_{k\in\N\setminus D}k^{-2}<\infty,\]
the convergence part of Theorem \ref{t:main} implies that for $\mu$-almost every $ \bx $,
\[\sum_{k\in\N\setminus D}\mathbbm{1}_{R(\bx,\br_k)}(T^k\bx)<\infty.\]
For any such $ \bx $, it follows that
\begin{align}
	&\lim_{n\to\infty}\dfrac{\sum_{k\in D_n}\mathbbm{1}_{R(\bx,\br_k)}(T^k\bx)}{\sum_{k\in D_n}2^d\gamma_k}=h(\bx)\label{eq:subset}\\
	\Longleftrightarrow\quad&\lim_{n\to\infty}\dfrac{\sum_{k=1}^n\mathbbm{1}_{R(\bx,\br_k)}(T^k\bx)}{\sum_{k=1}^n2^d\gamma_k}=h(\bx)\notag.
\end{align}
Then, the divergence part of Theorem \ref{t:main} holds provided that \eqref{eq:subset} holds for $ \mu $-almost every $ \bx $.

Given that the density $h$ of $\mu$ may be unbounded, it is more appropriate to study the following auxiliary sets that inspired by \cite{KKP21}. For any $ \bx\in[0,1]^d $ and $ n\in\N $, let $ l_n(\bx)\in\R_{\ge 0} $ be the non-negative number such that $ \mu(R(\bx,l_n(\bx)\br_n))=\gamma_n $. This is possible, since $\mu$ is non-singular. Let
\[\xi_n(\bx):=l_n(\bx)\br_n\in(\R_{\ge 0})^d.\]
Then $ R(\bx, \xi_n(\bx)) $ is a hyperrectangle obtained by scaling $ R(\bx, \br_n) $ by a factor $ l_n(\bx) $. Define
\[\hat E_n:=\{\bx\in[0,1]^d:T^n\bx\in R(\bx, \xi_n(\bx))\}.\]

The key ingredient for the proof of the divergence case is the utilise of quantitative form of the divergence Borel-Cantelli lemma.
\begin{thm}[{\cite[Chapter I, Lemma 10]{Spr79}}]\label{t:sbc}
	Let $ \{A_n\} $ be a sequence of measurable sets in a probability space $ (X,\nu) $. Denote by $ A(N,x) $ the number of integers $ n\le N $ such that $ x\in A_n $. Put
	\[\phi(N)=\sum_{n=1}^N\nu(A_n).\]
	Suppose that there exists a constant $ c $ such that for any $ 1\le M<N $,
	\[\sum_{M\le m<n\le N}\big(\nu(A_{n}\cap A_m)-\nu (A_n)\nu(A_m)\big)\le \sum_{n=M}^Nc\nu(A_n).\]
	Then for any $ \varepsilon>0 $ one has
	\[A(N,x)=\phi(N)+O_\varepsilon\big(\phi^{1/2}(N)\log^{3/2+\varepsilon}\phi(N)\big)\]
	for $ \nu $-almost every $ x $; in particular, if $ \sum_{n=1}^\infty\nu(A_n)=\infty $, then
	\[\lim_{N\to\infty}\frac{A(N,x)}{\phi(N)}=1\quad\text{for $ \nu $-a.e.\,$ x $}.\]
\end{thm}
To apply Theorem \ref{t:sbc}, it is essential to estimate both the measure of each $ \hat E_n $ and the correlations among the sets $ \{\hat E_n\} $. These aspects will be addressed in the upcoming two subsections, respectively.

\subsection{Lower estimate the measure of $ \hat E_n $}
By \eqref{eq:assr=0} we may assume that $ \gamma_n\ge n^{-2} $. For otherwise by \eqref{eq:assr=0} we would have $ \gamma_n=0 $, and so $ \mu(\hat E_n)=0 $. The following two lemmas coming from [18] are useful in estimating the $\mu$-measure of $ \hat E_n $. The first one describes some local structures of $ \hat E_n $; the other one calculates the $\mu$-measure of hyperrectangles related to $ \hat E_n $.

\begin{lem}\label{l:simple lemma s}
	Let $ B(\bx,r) $ be a ball with center $ \bx $ and $ r>0 $. Then, for any $ n\in\N $ with $ \gamma_n\ne 0 $ and any $ F\subset B(\bx,r) $, we have
	\[F\cap T^{-n}R(\bx,\xi_n(\bx)-2rn^{2}\br_n)\subset F\cap\hat E_n\subset F\cap T^{-n} R(\bx,\xi_n(\bx)+2rn^2\br_n).\]
\end{lem}
\begin{lem}\label{l:measure of enlarged ball}
	For any $ \br\in(\R_{\ge 0})^d $, we have
	\begin{align*}
		\mu\big(R(\bx,\xi_n(\bx)+\br)\big)&\le \gamma_n+2d|h|_q|\br|^s,\\
		\mu\big(R(\bx,\xi_n(\bx)-\br)\big)&\ge \gamma_n-2d|h|_q|\br|^s.
	\end{align*}
\end{lem}
Since $0<\eta<\tau_1<\tau$, here and hereafter, we will replace $ \tau $ and $ \tau_1 $ in Lemmas \ref{l:RcapF} and \ref{l:JcRcR} with $ \eta $, and the estimates still hold. Recall that a ball here is with respect to the maximum norm and thus corresponds to a Euclidean hypercube. The geometry of $ [0,1]^d $ allows us to write it as the union of $ n^{2d+4d/s} $ pairwise disjoint balls with common radius $ n^{-2-4/s} $. The collection of these balls is denoted by
\begin{equation}\label{eq:Bn}
	\mathcal B_n:=\{B(\bx_i,n^{-2-4/s}):1\le i\le n^{2d+4d/s}\}.
\end{equation}
Since $ n^{2+4/s} $ is in general non-integer we should use $ r=1/[n^{2+4/s}] $ instead of $ n^{-2-4/s} $ in the definition of $\mathcal B_n$ so that $[0,1]^d$ can be partitioned into $1/r$ balls of radius $r$, where $ [\cdot] $ denotes the integer part of a real number. However, to improve the readability we let the
correct definition be implicitly understood throughout the proof. The outcome is invariant under
this abuse of notation.

Now, we estimate the $\mu$-measure of $ \hat E_n $.
	\begin{lem}\label{l:measure of hat En}
	There exists a constant $ c_4 $ such that
	\[\mu(\hat E_n)\ge(1-c_4n^{-2})\gamma_n.\]
\end{lem}
\begin{proof}
		For each ball $ B_{n,i}=B(\bx_i,n^{-2-4/s})\in\mathcal B_n $, by Lemma \ref{l:simple lemma s} we have
	\[B_{n,i}\cap\hat E_n\supset B_{n,i}\cap T^{-n}R(\bx_i,\xi_n(\bx_i)-2n^{-2-4/s}n^2\br_n).\]
	Hence,
	\begin{equation}\label{eq:B(x,r) cap Em cap En}
		\hat E_n\supset \bigcup_{B_{n,i}\in\mathcal B_n}B_{n,i}\cap  T^{-n}R(\bx_i,\xi_n(\bx_i)-2n^{-4/s}\br_n).
	\end{equation}
	Applying Lemmas \ref{l:RcapF} and \ref{l:measure of enlarged ball}, we have
	\begin{align*}
		&\mu\big(B_{n,i}\cap  T^{-n}R(\bx_i,\xi_n(\bx_i)-2n^{-4/s}\br_n)\big)\notag\\
		\ge& (\mu(B_{n,i})-c_1e^{-\eta n})\mu\big(R(\bx_i,\xi_n(\bx_i)-2n^{-4/s}\br_n)\big)\\
		\ge&(\mu(B_{n,i})-c_1e^{-\eta n})(\gamma_n-4d|h|_qn^{-4}).
	\end{align*}
Summing over all $B_{n,i}\in\mathcal B_{n} $, we deduce that the $ \mu $-measure of $ \hat E_n $ is at least
	\[(1-n^{2d+4d/s}\cdot c_1e^{-\eta n})(\gamma_n-4d|h|_qn^{-4}).\]
	Since
	\[n^{2d+4d/s}e^{-\eta n}\ll n^{-2}\quad\text{and}\quad n^{-4}\le n^{-2}\gamma_n,\]
	we conclude that there exists a constant $ c_4 $ such that
	\[\mu(\hat E_n)\ge (1-c_4n^{-2})\gamma_n.\qedhere\]
\end{proof}
\subsection{Estimate the correlations of $ \hat E_m\cap \hat E_n $ with $ m< n $}
\begin{lem}\label{l:estimations on correlations 1}
	There exists a constant $ c_5 $ such that for any integers $ m $ and $ n $ with $ \sqrt n\le m<n $, we have
	\begin{equation*}
		\mu(\hat E_m\cap\hat E_n)\le (1+c_5n^{-2})\gamma_m\gamma_n+c_5e^{-\eta(n-m)}\gamma_n.
	\end{equation*}
\end{lem}
\begin{proof}
	Since $ m<n $, we have
	\[n^{-2-4/s}m^2\le n^{-2-4/s}n^2= n^{-4/s}.\]
	For each ball $ B_{n,i}=B(\bx_i,n^{-2-4/s})\in\mathcal B_n $, by Lemma \ref{l:simple lemma s} we have
	\begin{align}
		&B_{n,i}\cap \hat E_m\cap \hat E_n\notag\\
		\subset& B_{n,i}\cap T^{-m}R(\bx_i,\xi_m(\bx_i)+2n^{-4/s}\br_m)\cap T^{-n}R(\bx_i,\xi_n(\bx_i)+2n^{-4/s}\br_n),\label{eq:BRR}
	\end{align}
	and so
	\begin{equation*}\label{eq:EcE}
		\begin{split}
			&\hat E_m\cap\hat E_n\\
			\subset& \bigcup_{B_{n,i}}B_{n,i}\cap T^{-m}R(\bx_i,\xi_m(\bx_i)+2n^{-4/s}\br_m)\cap T^{-n}R(\bx_i,\xi_n(\bx_i)+2n^{-4/s}\br_n).
		\end{split}
	\end{equation*}
	Applying Lemma \ref{l:RcapF} twice to the sets in \eqref{eq:BRR}, we have
	\begin{align}
		&\mu\big(B_{n,i}\cap T^{-m}R(\bx_i,\xi_m(\bx_i)+2n^{-4/s}\br_m)\cap T^{-n}R(\bx_i,\xi_n(\bx_i)+2n^{-4/s}\br_n)\big)\notag\\
		\le&\big(\mu(B_{n,i})+c_1e^{-\eta m}\big)\mu\big(R(\bx_i,\xi_m(\bx_i)+2n^{-4/s}\br_m)\cap T^{-(n-m)}R(\bx_i,\xi_n(\bx_i)+2n^{-4/s}\br_n)\big)\notag\\
		\le& \big(\mu(B_{n,i})+c_1e^{-\eta m}\big)\big(\mu(R(\bx_i,\xi_m(\bx_i)+2n^{-4/s}\br_m))+c_1e^{-\eta(n-m)}\big)\notag\\
		&\hspace{18em}\cdot \mu(R(\bx_i,\xi_n(\bx_i)+2n^{-4/s}\br_n)).\notag
	\end{align}
	By Lemma \ref{l:measure of enlarged ball}, the last quantity is majorized by
	\[ \big(\mu(B_{n,i})+c_1e^{-\eta m}\big)(\gamma_m+4d|h|_qn^{-4}+c_1e^{-\eta(n-m)})(\gamma_n+4d|h|_qn^{-4}).\]
	Summing over $ B_{n,i}\in\mathcal B_n $, we have
	\[\begin{split}
		&\mu(\hat E_m\cap \hat E_n)\\
		\le& \big(1+c_1n^{2d+4d/s}e^{-\eta m}\big)(\gamma_m+4d|h|_qn^{-4}+c_1e^{-\eta(n-m)})(\gamma_n+4d|h|_qn^{-4}).
	\end{split}\]
	From the conditions $ \gamma_m\ge m^{-2} $, $ \gamma_n\ge n^{-2} $ and $ \sqrt n\le m<n  $, it follows that
	\[4d|h|_qn^{-4}\le cn^{-2}\max (\gamma_m,\gamma_n)\qaq c_1n^{2d+4d/s}e^{-\eta m}\le c n^{-2}\]
	for some absolute constant $ c>0 $.
	Thus, we have
	\[\begin{split}
			&\mu(\hat E_m\cap \hat E_n)\\
			\le &\big(1+cn^{-2}\big)(\gamma_m(1+cn^{-2})+c_1e^{-\eta(n-m)})\cdot (1+cn^{-2})\gamma_n\\
			= &(1+cn^{-2})^3\gamma_m \gamma_n+(1+cn^{-2})^2c_1e^{-\eta (n-m)}\gamma_n.
		\end{split}\]
	By choosing a proper constant $c_5$, we arrive at the conclusion.
\end{proof}
The correlation estimates presented in Lemma \ref{l:estimations on correlations 1} work effectively for the case $ m\ge \sqrt{n} $, and the proof relies solely on Conditions I and II. We will employ Lemma \ref{l:measure of intersection} to estimate the correlations of $ \hat E_m\cap \hat{E}_n $ for $ m<\sqrt n $.

\begin{lem}\label{l:estimations on correlations 2}
	There exists a constant $ c_6 $ such that for any integers $ m $ and $ n $ with $ 0<m<\sqrt n $, we have
	\[\mu(\hat E_m\cap \hat E_n)\le(1+c_6m^{-2}) \gamma_m\gamma_n.\]
\end{lem}
\begin{proof}
	Recall the definitions of $D(m,n)$ and $\mathcal F_m(n)$ in Lemma \ref{l:measure of intersection}. Since $D(m,n)\cup\mathcal \bigcup_{J_m\in\mathcal F_m(n)}J_m=[0,1]^d$, by Lemma \ref{l:measure of intersection}
	\begin{align}
		\mu(\hat E_m\cap \hat E_n)&\le \mu(D(m,n)\cap\hat E_m\cap \hat E_n)+\sum_{J_m\in\mathcal F_m(n)}\mu(J_m\cap \hat E_m\cap \hat E_n)\notag\\
		&\le c_3n^{-6}+\sum_{J_m\in\mathcal F_m(n)}\mu(J_m\cap \hat E_m\cap \hat E_n).\notag\\
		&\le c_3n^{-2}\gamma_m\gamma_n+\sum_{J_m\in\mathcal F_m(n)}\mu(J_m\cap \hat E_m\cap \hat E_n).\label{eq:twosum}
	\end{align}

	Now, we estimate the summation in \eqref{eq:twosum}. By decreasing the radius of balls in $ \mathcal B_n $ by a factor between $ 1/2 $ and $ 1 $, we can assume that each ball in $ \mathcal B_n $ is contained in a unique ball in $ \mathcal B_m $. This manipulation does affect the later conclusion. By Lemma \ref{l:simple lemma s} we have
	\[\begin{split}
		&B_{m,j}\cap\hat E_m\cap\hat E_n=\bigcup_{B_{n,i}:B_{n,i}\subset B_{m,j}}(B_{m,j}\cap\hat E_m)\cap (B_{n,i}\cap\hat E_n)\\
		\subset& \bigcup_{B_{n,i}:B_{n,i}\subset B_{m,j}}  B_{n,i}\cap T^{-m}R(\by_j,\xi_m(\by_j)+2m^{-4/s}\br_m)\cap  T^{-n}R(\bx_i,\xi_n(\bx_i)+2n^{-4/s}\br_n),
	\end{split}\]
	where $ B_{m,j}=B(\by_j,m^{-2-4/s})\in\mathcal B_m $ and $ B_{n,i}=B(\bx_i,n^{-2-4/s})\in\mathcal B_n $. Hence,
	\begin{equation}\label{eq:B(x,r) cap Em cap En s}
		\begin{split}
			&\bigcup_{J_m\in \mathcal F_m(n)}J_m\cap\hat E_m\cap\hat E_n\\
			&\subset \bigcup_{B_{m,j}\in\mathcal B_m}\bigcup_{B_{n,i}:B_{n,i}\subset B_{m,j}}\bigcup_{J_m\in \mathcal F_m(n)}J_m\cap G_{i,j}(m)\cap T^{-n}R(\bx_i,\xi_n(\bx_i)+2n^{-4/s}\br_n),
		\end{split}
	\end{equation}
	where
	\begin{equation}\label{eq:Gim}
		G_{i,j}(m)=B_{n,i}\cap T^{-m}R(\by_j,\xi_m(\by_j)+2m^{-4/s}\br_m).
	\end{equation}
	By \eqref{eq:B(x,r) cap Em cap En s} and Lemmas \ref{l:measure of intersection} and \ref{l:measure of enlarged ball} we have
	\begin{align}
		&\sum_{J_m\in\mathcal F_m(n)}\mu(J_m\cap \hat E_m\cap\hat E_n)\notag\\
		\le &\sum_{B_{m,j}\in\mathcal B_m}\sum_{B_{n,i}\subset B_{m,j}}\sum_{J_m\in\mathcal F_m(n)}\mu\big(J_m\cap G_{i,j}(m)\cap T^{-n}R(\bx_i,\xi_n(\bx_i)+2n^{-4/s}\br_n)\big)\notag\\
		\le&\sum_{B_{m,j}\in\mathcal B_m}\sum_{B_{n,i}\subset B_{m,j}} \big(\mu(G_{i,j}(m))+c_5e^{-\eta   n}\big)\mu\big(R(\bx_i,\xi_n(\bx_i)+2n^{-4/s}\br_n)\big)\notag\\
		\le&\sum_{B_{m,j}\in\mathcal B_m}\sum_{B_{n,i}\subset B_{m,j}} \big(\mu(G_{i,j}(m))+c_5e^{-\eta n}\big)(\gamma_n+4d|h|_qn^{-4})\label{eq:summation need to estimate}.
	\end{align}
	Using \eqref{eq:Gim} and Lemma \ref{l:RcapF}, we deduce that
	\begin{align}
		&\sum_{B_{m,j}\in\mathcal B_m}\sum_{B_{n,i}\subset B_{m,j}} \big(\mu(G_{i,j}(m))+c_5e^{-\eta n}\big)\notag\\
		= & \sum_{B_{m,j}\in\mathcal B_m}\mu\big(B_{m,j}\cap T^{-m}R(\by_j,\xi_m(\by_j)+2m^{-4/s}\br_m)\big)+c_5n^{2d+4d/s}e^{-\eta n}\notag\\
		\le &\sum_{B_{m,j}\in\mathcal B_m}\big(\mu(B_{m,j})+c_1e^{-\eta m}\big)\mu\big(R(\by_j,\xi_m(\by_j)+2m^{-4/s}\br_m)\big)+c_5n^{2d+4d/s}e^{-\eta n}\notag\\
		\le &(1+c_1m^{2d+4d/s}e^{-\eta m})(\gamma_m+4d|h|_qm^{-4})+c_5n^{2d+4d/s}e^{-\eta n}.\label{eq:split}
	\end{align}
	Since
	\[c_5n^{2d+4d/s}e^{-\eta n}\ll c_1m^{2d+4d/s}e^{-\eta m}\le cm^{-2}\gamma_m,\]
	and
	\[4d|h|_qm^{-4}\le c m^{-2}\gamma_m\quad \text{and}\quad 4d|h|_qn^{-4}\le c n^{-2}\gamma_n\]
	for some constant $ c>0 $, by \eqref{eq:summation need to estimate} and \eqref{eq:split}, we have
	\[\begin{split}
		\sum_{J_m\in\mathcal F_m(n)}\mu(J_m\cap \hat E_m\cap\hat E_n)\le &\big( (1+cm^{-2})^2\gamma_m+cm^{-2}\gamma_m\big)(1+cn^{-2})\gamma_n.
	\end{split}\]
	Therefore, there exists a constant $ c_6 $ such that
	\[\mu(\hat E_m\cap \hat E_n)\le(1+c_6m^{-2}) \gamma_m\gamma_n.\qedhere\]
\end{proof}
\subsection{Completing the proof of Theorem \ref{t:main}}
We first show the strong dynamical Borel-Cantelli lemma for the auxilliary sets $ \hat E_n $ and then using Zygmund differentiation theorem to complete the proof of the divergence part of Theorem \ref{t:main}.
\begin{lem}\label{l:sbchat}
	For $ \mu $-almost every $ \bx $, we have
	\begin{equation}\label{eq:sbchat}
		\lim_{n\to\infty}\frac{\sum_{k=1}^n\mathbbm{1}_{R(\bx, \xi_k(\bx))}(T^k\bx)}{\sum_{k=1}^n\mu\big(R(\bx,\xi_k(\bx))\big)}=1.
	\end{equation}
\end{lem}
\begin{proof}
	Let $ 1\le M<N $, we have
	\begin{align}
		&\sum_{M\le m<n\le N}\big(\mu(\hat E_m\cap \hat E_n)-\mu(\hat E_n)\mu(\hat E_m)\big)\notag\\
		=&\bigg(\sum_{M\le m<n\le N\atop m\ge \sqrt{n}}+\sum_{M\le m<n\le N\atop m<\sqrt{n}}\bigg)\big(\mu(\hat E_m\cap \hat E_n)-\mu(\hat E_n)\mu(\hat E_m)\big)\label{eq:corsum}
	\end{align}
By Lemmas \ref{l:measure of hat En} and \ref{l:estimations on correlations 1}, the first summation can be estimated as
\begin{align*}
	\le&\sum_{M\le m<n\le N\atop m\ge \sqrt{n}} \big((1+c_5n^{-2})\gamma_m\gamma_n+c_5e^{-\eta(n-m)}\gamma_n-(1-c_4m^{-2})(1-c_4n^{-2})\gamma_m\gamma_n\big)\\
	\ll &\sum_{M\le m<n\le N\atop m\ge \sqrt{n}} \big(n^{-2}\gamma_n+e^{-\eta(n-m)}\gamma_n+m^{-2}\gamma_n\big)\ll\sum_{n=M}^{N}\gamma_n\asymp\sum_{n=M}^{N}\mu(\hat E_n).
\end{align*}
For the second summation in \eqref{eq:corsum}, we have
\begin{align*}
	\le&\sum_{M\le m<n\le N\atop m< \sqrt{n}} \big((1+c_6 m^{-2}) \gamma_m\gamma_n-(1-c_4m^{-2})(1-c_4n^{-2})\gamma_m\gamma_n\big)\\
	\ll &\sum_{M\le m<n\le N\atop m< \sqrt{n}} m^{-2}\gamma_n\ll\sum_{n=M}^{N}\gamma_n\asymp\sum_{n=M}^{N}\mu(\hat E_n).
\end{align*}
Now, apply Theorem \ref{t:sbc}, we finish the proof.
\end{proof}

\begin{thm}[Zygmund differentiation theorem {\cite[Theorem 2.29]{Hajlasz}}]\label{t:differentiation theorem}
	Let $ \{\br_n\} $ be a sequence of positive vectors with $ \lim_{n\to\infty}|\br_n|\to 0 $. If $ f\in L^q(\lm^d) $ for some $ q>1 $, then
	\[\lim_{n\to\infty}\frac{\int_{R(\bx,\br_n)} f(\by)\dif\lm^d(\by)}{\lm^d\big(R(\bx,\br_n)\big)}=f(\bx)\qquad\text{for }\lm^d\text{-a.e.\,}\bx.\]
\end{thm}
Now, we employ Zygmund differentiation theorem to complete the proof.
\begin{proof}[Proof of Theorem \ref{t:main}: divergence part]
	Let $ \{t_i\} $ be an enumeration of positive rational numbers. For each $ i $, define a new sequence $ \{\br^{(i)}_n\}_{n\ge 1} $ with $ \br^{(i)}_n=t_i\br_n $. Let $ \xi_n^{(i)}(\bx) $ be defined in the same way as $ \xi_n(\bx) $, that is, $ \xi_n^{(i)}(\bx)=l_n^{(i)}(\bx)\br_n $ is chosen to satisfy $ \mu(R(\bx,\xi_n^{(i)}(\bx)))=t_i^d\gamma_n $, where $l_n^{(i)}(\bx)\ge 0$. Since each $ t_i $ is positive, for any $i$ and the sequence $\{\br_n^{(i)}\}_{n\ge 1}$, we can  deduce from Lemma \ref{l:sbchat} that \eqref{eq:sbchat} holds for $\mu$-almost every $ \bx $ with $ \xi_k(\bx) $ replaced by $ \xi_k^{(i)}(\bx) $. Let $ F $ be the intersection of the set of points satisfying \eqref{eq:sbchat} with $ \xi_k(\bx) $ replaced by $ \xi_k^{(i)}(\bx) $ (for all $i$) and the set of points satisfying Theorem \ref{t:differentiation theorem}. Since $\{t_i\}$ is countable, $ F $ is a countable intersection of sets with full $\mu$-measure, and so $F$ itself is also of full $\mu$-measure.

	Let $ \bx\in F $. For any $ \varepsilon>0 $, by the density of rationals, we can choose $ t_i $ and $t_j $ such that
	\begin{equation}\label{eq:tij}
		t_i^d< 2^dh(\bx)< t_i^d(1+\varepsilon) \qquad\text{and}\qquad t_j^d(1-\varepsilon)< 2^dh(\bx)< t_j^d.
	\end{equation}
	By the definition of $ F $, it follows from Theorem \ref{t:differentiation theorem} that
	\[\lim_{n\to\infty}\frac{\mu\big(R(\bx,\br_n)\big)}{2^d\gamma_n}=h(\bx)\quad\text{or equivalently}\quad \lim_{n\to\infty}\frac{\mu\big(R(\bx,\br_n)\big)}{\gamma_n}=2^dh(\bx).\]
	Since $t_i^d<2^dh(\bx)$ (see \eqref{eq:tij}) and $ \mu(R(\bx,\xi_n^{(i)}(\bx)))=\mu(R(\bx,l_n^{(i)}(\bx)\br_n))=t_i^d\gamma_n $, we have $ l_n^{(i)}(\bx)\le 1 $ for all large $ n $. Similarly, we have $ l_n^{(j)}(\bx)\ge 1 $ for all large $ n $. Therefore, for all large $n$,
	\[R(\bx,\xi_n^{(i)}(\bx))\subset R(\bx,\br_n)\subset R(\bx,\xi_n^{(j)}(\bx)).\]
	This together with Lemma \ref{l:sbchat} and \eqref{eq:tij} gives
	\begin{equation*}
		\begin{split}
			\frac{1}{1+\varepsilon}=\lim_{n\to\infty}\frac{\sum_{k=1}^n\mathbbm{1}_{R(\bx, \xi_k^{(i)}(\bx))}(T^k\bx)}{(1+\varepsilon)\sum_{k=1}^n\mu\big(R(\bx,\xi_k^{(i)}(\bx))\big)}&=\lim_{n\to\infty}\frac{\sum_{k=1}^n\mathbbm{1}_{R(\bx, \xi_k^{(i)}(\bx))}(T^k\bx)}{\sum_{k=1}^n(1+\varepsilon)t_i^d\gamma_k}\\
			&\le \lim_{n\to\infty}\frac{\sum_{k=1}^n\mathbbm{1}_{R(\bx, \br_k)}(T^k\bx)}{\sum_{k=1}^n2^dh(\bx)\gamma_k}.
		\end{split}
	\end{equation*}
	Similarly, we have
	\begin{equation*}
		\begin{split}
			\lim_{n\to\infty}\frac{\sum_{k=1}^n\mathbbm{1}_{R(\bx, \br_k)}(T^k\bx)}{\sum_{k=1}^n2^dh(\bx)\gamma_k}\le\lim_{n\to\infty}\frac{\sum_{k=1}^n\mathbbm{1}_{R(\bx, \xi_k^{(j)}(\bx))}(T^k\bx)}{\sum_{k=1}^n(1-\varepsilon)t_j^d\gamma_k}= \frac{1}{1-\varepsilon}.
		\end{split}
	\end{equation*}
	By the arbitrariness of $ \varepsilon $, we finish the proof.
\end{proof}


\begin{thebibliography}{10}

	\bibitem{ABB22}
	D. Allen, S. Baker and B. B\'ar\'any. Recurrence rates for shifts of finite type. {\em Preprint} (2022), 	arXiv:2209.01919.



	\bibitem{BakerFarmer21}
S. Baker and M. Farmer. Quantitative recurrence properties for self-conformal sets. {\em Proc. Amer. Math. Soc.} 149 (2021), 1127--1138.

\bibitem{BK24}
S. Baker and H. Koivusalo. Quantitative recurrence and the shrinking target problem for overlapping iterated function systems. {\em Adv. Math.} 442, Paper No. 109538, 65 pp.

\bibitem{Baladi00}
V. Baladi. {\em Positive Transfer Operators and Decay of Correlations}. World Scientific, River Edge, NJ, 2000.

	\bibitem{BarreiraSaussol01}
	L. Barreira and B. Saussol. Hausdorff dimension of measures via {P}oincar\'{e} recurrence. {\em   Comm. Math. Phys.} 219 (2001), 443--463.

	\bibitem{Boshernitzan93}
M. D. Boshernitzan. Quantitative recurrence results. {\em Invent. Math.} 113 (1993), 617--631.

	\bibitem{CWW19}
	Y. Chang, M. Wu and W. Wu. Quantitative recurrence properties and homogeneous self-similar sets. {\em Proc. Amer. Math. Soc.} 147 (2019), 1453--1465.

	\bibitem{ChKl01}
	N. Chernov and D. Kleinbock. Dynamical Borel-Cantelli lemmas for Gibbs measures. {\em Israel J. Math.} 122 (2001), 1--27.

		 \bibitem{Fu81}
	H. Furstenberg. {\em Recurrence in ergodic theory and combinatorial number theory.} Princeton University Press, Princeton, NJ, 1981.

	\bibitem{Hajlasz}
	P. Haj\l asz. {\em Harmonic analysis.} https://sites.pitt.edu/\textasciitilde hajlasz/Notatki/Harmonic\%20Analysis4.pdf, 2017.

	\bibitem{HeLi24}
	Y. He and L. Liao. Quantitative recurrence properties for piecewise expanding map on $ [0,1]^d $. {\em Ann. Sc. Norm. Super. Pisa Cl. Sci.} Published online. DOI: 10.2422/2036-2145.202309\_008.

	\bibitem{HLSW22}
	M. Hussain, B. Li, D. Simmons and B. Wang. Dynamical {B}orel-{C}antelli lemma for recurrence theory. {\em Ergodic Theory Dynam. Systems} 42 (2022), 1994--2008.
	\bibitem{Kim07}
	D. H. Kim. The dynamical Borel-Cantelli lemma for interval maps. {\em Discrete Contin.
	Dynam. Systems} 17 (2007), 891--900.

	\bibitem{KKP21}
	M. Kirsebom, P. Kunde and T. Persson. On shrinking targets and self-returning points. {\em Ann. Sc. Norm. Super. Pisa Cl. Sci. (5)} 24 (2023), 1499--1535.

	\bibitem{KleinbockZheng22}
	D. Kleinbock and J. Zheng. Dynamical Borel-Cantelli lemma for recurrence under Lipschitz twists. {\em Nonlinearity} 36 (2023), 1434--1460.

	\bibitem{LLVZ23}
	B. Li, L. Liao, S. Velani and E. Zorin. The shrinking target problem for matrix transformations of tori: revisiting the standard problem. {\em Adv.Math.} 421 (2023), Paper No. 108994, 74.

	\bibitem{Per23}
	T. Persson. A strong Borel-Cantelli lemma for recurrence. {\em Studia Math.} 268 (2023), 75--89.

	\bibitem{Phi67}
	W. Philipp. Some metrical theorems in number theory. {\em Pacific J. Math.} 20 (1967), 109--127.

	\bibitem{Spr79}
	V. Sprind\v zuk. {\em Metric theory of Diophantine approximations.} John Wiley \& Sons,
	New York, Toronto, London, 1979.

	\bibitem{Spo23}
	A. R. Sponheimer. A recurrence-type strong Borel-Cantelli lemma for Axiom A diffeomorphisms. {\em Preprint} (2023),  arXiv:2307.12928.

\end{thebibliography}
\end{document}